\documentclass[a4paper]{article}
\def\ColourDiagrams{---}
\usepackage{amssymb,amsmath,amsfonts,amsthm,graphicx,verbatim,enumerate}
\usepackage{hyperref}

\theoremstyle{plain}
\newtheorem{proposition}{Proposition}[section]
\newtheorem{theorem}[proposition]{Theorem}
\newtheorem{corollary}[proposition]{Corollary}

\newtheorem{lemma}[proposition]{Lemma}
\newtheorem{conjecture}[proposition]{Conjecture}

\def\msig{\mc{F}^\sigma_m}
\def\mc{\mathcal}

\def\F{\mathcal{F}}
\def\tr{\textrm}

\begin{document}

\title{A solution to the 2/3 conjecture}

\author{Rahil Baber
\and John Talbot\thanks{Department of Mathematics, UCL, London, WC1E
6BT, UK. Email: j.talbot@ucl.ac.uk.  }}

\date\today

\maketitle

\begin{abstract}
We prove a vertex domination conjecture of Erd\H os, Faudree, Gould,
Gy\'arf\'as, Rousseau, and Schelp, that for every $n$-vertex
complete graph with edges coloured using three colours there exists
a set of at most three vertices which have at least $2n/3$
neighbours in one of the colours. Our proof makes extensive use of the ideas
presented in ``A New Bound for the 2/3 Conjecture'' by Kr\'al', Liu,
Sereni, Whalen, and Yilma.
\end{abstract}

\section{Introduction}

In this paper we prove the $2/3$ conjecture of Erd\H os, Faudree,
Gould, Gy\'arf\'as, Rousseau, and Schelp \cite{Erdos_22Set}. Before
we discuss this problem we first require some definitions.

A \emph{graph} is a pair of sets $G = (V(G),E(G))$ where
$V(G)$ is the set of \emph{vertices} and $E(G)$ is a family of
$2$-subsets of $V(G)$ called \emph{edges}. A \emph{complete graph}
is a graph containing all possible edges.

An \emph{$r$-colouring} of the edges of a graph $G$ is a map from
$E(G)$ to a set of size $r$. Given an $r$-colouring of the edges of
a complete graph $G=(V(G),E(G))$, a colour $c$, and $A,B\subseteq
V(G)$, we say that $A$ \emph{$c$-dominates} $B$ if for every $b\in
B\setminus A$ there exists $a\in A$ such that the edge $ab$ is
coloured $c$. We say that $A$ \emph{strongly $c$-dominates} $B$ if
for every $b\in B$ there exists an $a\in A$ such that the edge $ab$
is coloured $c$. Note that if $A$ strongly $c$-dominates $B$ then it
also
$c$-dominates $B$.

Erd\H os and Hajnal \cite{Hajnal} showed that given a positive integer $t$, a real value
$\epsilon>0$, and a $2$-coloured complete graph on $n>n_0$ vertices,
there exists a set of $t$ vertices that $c$-dominate at least
$(1-(1+\epsilon)(2/3)^t)n$ vertices for some colour $c$. They asked
whether $2/3$ could be replaced by $1/2$, which was answered by
Erd\H os, Faudree, Gy\'arf\'as, and Schelp \cite{Erdos_2Colour}.

\begin{theorem}[Erd\H os, Faudree, Gy\'arf\'as, and Schelp
\cite{Erdos_2Colour}]\label{THM:2Col} For any positive integer $t$,
and any $2$-coloured complete graph on $n$ vertices, there exists a
colour $c$ and a set of at most $t$ vertices that $c$-dominate at
least $(1-1/2^t)n$ of the vertices.
\end{theorem}

\begin{figure}[tbp]
\begin{center}
\ifdefined\ColourDiagrams
\includegraphics[height=5cm]{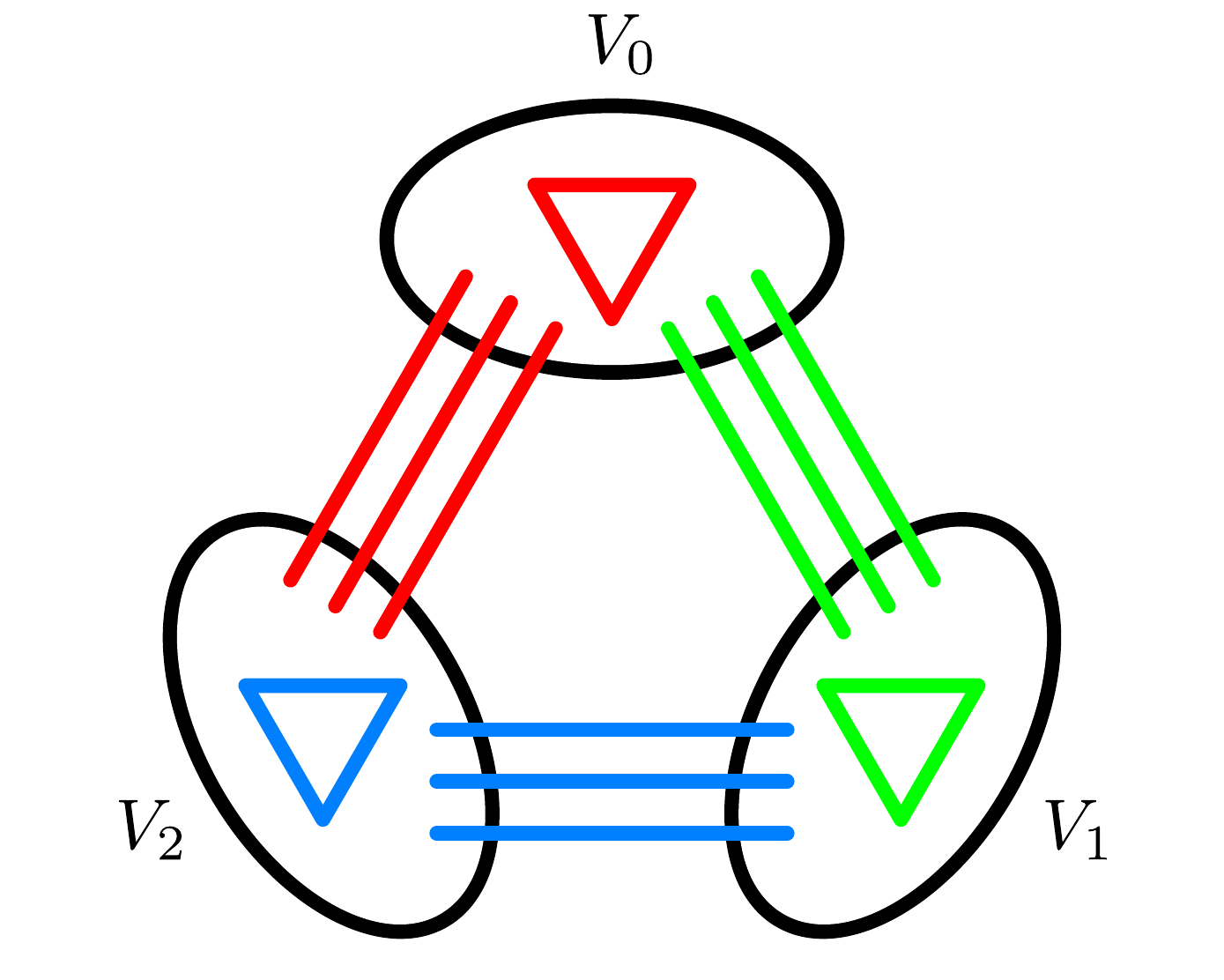}
\caption{Kierstead's construction. The colours $0,1,2$ are
represented by red, green, and blue respectively.}
\else
\includegraphics[height=5cm]{Extremal3_grey}
\caption{Kierstead's construction. The colours $0,1,2$ are
represented by solid black, dashed black, and solid grey lines
respectively.} \fi \label{FIG:Extremal3}
\end{center}
\end{figure}

Erd\H os, Faudree, Gy\'arf\'as, and Schelp went on to ask whether
their result could be generalised to say that all $r$-coloured
complete graphs contain a set of $t$ vertices that $c$-dominate at
least $(1-(1-1/r)^t)n$ of the vertices. However, in the same paper
\cite{Erdos_2Colour}, they presented a construction by Kierstead
showing this to be false even for $r=3$ and $t=3$. Simply partition
the vertices of a complete graph into $3$ equal classes
$V_0,V_1,V_2$, and colour the edges such that an edge $xy$ with
$x\in V_i$ and $y\in V_j$ is coloured $i$ if $i=j$ or $i\equiv
j+1\bmod 3$, see Figure \ref{FIG:Extremal3}. The construction shows
that for $3$-colourings it is impossible for a small set of vertices
to monochromatically dominate significantly more than $2/3$ of the
vertices. (It also shows that regardless of the size of our
dominating set we cannot guarantee more than $\lceil 2n/3 \rceil$
vertices will be strongly $c$-dominated.) Motivated by this example
Erd\H os, Faudree, Gould, Gy\'arf\'as, Rousseau, and Schelp
\cite{Erdos_22Set}, made the following conjecture.

\begin{conjecture}[Erd\H os, Faudree, Gould, Gy\'arf\'as, Rousseau, and Schelp
\cite{Erdos_22Set}]\label{CONJ:3Set}  For any $3$-coloured complete
graph, there exists a colour $c$ and a set of at most $3$ vertices
that $c$-dominates at least $2/3$ of the vertices.
\end{conjecture}

They were able to show that the conjecture holds true when it was
relaxed to asking for a dominating set of at most $22$ vertices, but
were unable to reduce $22$ to $3$. We note that $3$ is best
possible because in a typical random 3-colouring of a complete graph of order $n$ no
pair of vertices will monochromatically dominate more than $5n/9 +
o(n)$ vertices (this follows simply from Chernoff's bound). For completeness we should also
mention that in \cite{Erdos_22Set} the authors showed there always exist
$2$ vertices that monochromatically dominate at least $5(n-1)/9$
vertices in a $3$-coloured complete graph.

Kr\'al', Liu, Sereni, Whalen and Yilma \cite{Kral_4Set}, made
significant progress with Conjecture \ref{CONJ:3Set}, by proving
that there exists a colour $c$ and set of size at most $4$ which not
only $c$-dominates but strongly $c$-dominates at least $2/3$ of the
vertices in a $3$-coloured complete graph. Their proof makes use of
Razborov's semidefinite flag algebra method \cite{RF} to show that
Kierstead's construction is essentially extremal. We will discuss
flag algebras in more detail in Section \ref{SEC:FlagAlgebras}.

We verify Conjecture \ref{CONJ:3Set} by proving the following
theorem.

\begin{theorem}\label{THM:MainResult}
For any $3$-colouring of the edges of a complete graph on $n\geq 3$
vertices, there exists a colour $c$ and a set of $3$ vertices that
strongly $c$-dominate at least $2n/3$ vertices.
\end{theorem}

Our proof builds on the work of Kr\'al', Liu, Sereni,
Whalen, and Yilma \cite{Kral_4Set}. The main difference is that by
using an idea of Hladky, Kr\'al', and Norine \cite{Kral_Digraphs} we
have additional constraints to encode the $2/3$ condition when
applying the semidefinite flag algebra method (see Lemma
\ref{LEM:GkFlagConstraint}). Another difference is that we conduct our computations on $6$
vertex graphs, whereas in \cite{Kral_4Set} they only look at $5$
vertex graphs.

\section{Proof of Theorem \ref{THM:MainResult}}\label{SEC:ThmProof}

For the remainder of this paper we will let $\hat{G}$ be a fixed counterexample to Theorem
\ref{THM:MainResult} with $|V(\hat{G})|=k$. So $\hat{G}$ is a $3$-coloured complete graph on
$k\geq 3$ vertices such that every set of $3$ vertices strongly
$c$-dominates strictly less than $2k/3$ vertices for each colour
$c$. We will show that $\hat{G}$ cannot exist by proving that it would have to
satisfy two contradicting properties.

Given a $3$-coloured complete graph and a vertex $v$, let $A_v$ denote the set of colours of the edges incident to $v$.

The following lemma is implicitly given in
the paper by Kr\'al', Liu, Sereni, Whalen, and Yilma
\cite{Kral_4Set}.
\begin{lemma}[Kr\'al', Liu, Sereni, Whalen, and Yilma \cite{Kral_4Set}]\label{LEM:HasRainbow}
Our counterexample $\hat{G}$ must contain a vertex $v\in V(\hat{G})$ with $|A_v|
= 3$.
\end{lemma}
\begin{proof}
Since $\hat{G}$ is a counterexample it cannot contain a vertex $v$ with
$|A_v|=1$ otherwise any set of $3$ vertices containing $v$ will
strongly $c$-dominate all the vertices, for $c\in A_v$. So it is
enough to show that if $|A_v|=2$ for every vertex then $\hat{G}$ is not a
counterexample.

Let the set of colours be
$\{1,2,3\}$. If every vertex has $|A_v|=2$ we can partition the
vertices into three disjoint classes $V_1,V_2,V_3$ where $v\in V_i$
if $i\notin A_v$. Without loss of generality we can assume
$|V_1|\geq |V_2| \geq |V_3|$. Note that the colour of all edges $uv$
with $u\in V_1$ and $v\in V_2$ is $3$ because $A_u\cap A_v =
\{2,3\}\cap\{1,3\} = \{3\}$. Consequently any set of $3$ vertices
containing a vertex from $V_1$ and a vertex from $V_2$ must strongly
$3$-dominate $V_1\cup V_2$ which is at least $2/3$ of the vertices.

To complete the proof we need to consider what happens if we cannot
choose a vertex from $V_1$ and $V_2$. This can only occur
if $V_2=\emptyset$ which implies $V_3=\emptyset$ and $V_1 = V(\hat{G})$,
i.e.\ $\hat{G}$ is $2$-coloured. In this case we can apply the result of
Erd\H os, Faudree, Gy\'arf\'as, and Schelp, Theorem \ref{THM:2Col}.
Although technically the theorem is not stated in terms of strongly
$c$-dominating a set, its proof given in \cite{Erdos_2Colour} is
constructive and it can be easily checked that the dominating set it
finds is strongly $c$-dominating (for $t\geq 2$ and $n\geq 2$). \end{proof}

We will show via the semidefinite flag algebra method that $\hat{G}$
cannot contain a vertex $v$ with $|A_v|=3$ contradicting Lemma
\ref{LEM:HasRainbow}. The flag algebra method is primarily used to
study the limit of densities in sequences of graphs. As such we will
not apply it directly to $\hat{G}$ but to a sequence of graphs
$(G_n)_{n\in\mathbb{N}}$ where $G_n$ is constructed from $\hat{G}$ as
follows. $G_n$ is a $3$-coloured complete graph on $nk$ vertices
where each vertex $u\in V(\hat{G})$ has been replaced by a class of
$n$ vertices $V_u$. The edges of $G_n$ are coloured as follows:
edges between two classes $V_u$ and $V_v$ have the same colour as
$uv$ in $\hat{G}$, while edges within a class, $V_u$ say, are
coloured independently and uniformly at random with the colours from
$A_u$.

We would like to claim that $G_n$ is also a counterexample, but this may not be true. However, there
exist particular types of $3$ vertex sets which with high
probability strongly $c$-dominate at most $2/3+o(1)$ of the vertices
in $G_n$ for some colour $c$. (Unless otherwise stated $o(1)$ will denote a quantity that tends to zero as $n\to \infty$.)

We note that Chernoff's bound implies that for all $u\in V(\hat{G})$, $c\in A_u$ and $v\in V_u\subset V(G_n)$ we have \[
|\{w\in V_u: vw\textrm{ is coloured }c\}|=\frac{n}{|A_u|}+o(n),\] with probability $1-o(1)$.

Given a $3$-coloured complete graph and a colour $c$ we define a
\emph{good set for $c$} to be a set of $3$ vertices $\{x,y,z\}$ such
that either
\begin{enumerate}
\item[(i)] at least two of the edges $xy, xz, yz$ are coloured $c$, or
\item[(ii)] one of the edges, $xy$ say, is coloured $c$ and the remaining
vertex $z$ satisfies $|A_z\cup\{c\}|=3$.
\end{enumerate}
(Although this definition does not appear particularly natural, it
has the advantage of being easily encoded by the semidefinite flag
algebra method.)

\begin{lemma}\label{LEM:GoodSet}
Any good set for $c$ in $G_n$
strongly $c$-dominates at most $2/3+o(1)$ of the vertices with
probability $1-o(1)$.
\end{lemma}

\begin{proof}
For $u\in V(\hat{G})$ recall that $V_u$ is the corresponding class
of $n$ vertices in $G_n$. Given $S = \{x,y,z\}$ a good set for $c$
in $G_n$, we will consider its ``preimage'' in $\hat{G}$ which we
denote $S'$, i.e.\ $S'\subseteq V(\hat{G})$ is minimal such that
$S\subseteq \bigcup_{u\in S'} V_u$. Let the strongly $c$-dominated
sets be $D_S$ and $D_{S'}$ for $S$ in $G_n$ and $S'$ in $\hat{G}$
respectively. For $v\in S$ there exists some $u\in V(\hat{G})$ such
that $v\in V_u$, let us define $W_v$ to be the vertices that lie
within $V_u$ that are strongly $c$-dominated by $v$.

We first consider the case where $|S'|=3$ (which occurs when no
two members of $S$ lie in the same vertex class). Since $\hat{G}$ is a
counterexample we have $|D_{S'}|<2k/3$ or equivalently $|D_{S'}|\leq
2k/3-1/3$. It is easy to check that
\[
D_S = \left(\bigcup_{d\in D_{S'}} V_d\right) \cup
\left(\bigcup_{v\in S} W_v\right).
\]
We can split the problem into two cases depending on which type of
good set $S$ is. If $S$ is of type (i) then it is
easy to see that $W_v\subseteq \bigcup_{d\in D_{S'}} V_d$ for every
$v\in S$. Consequently $D_S = \bigcup_{d\in D_{S'}} V_d$ implying
$|D_S|=n|D_{S'}| \leq 2nk/3-n/3$, hence $D_S$ contains at most $2/3$
of the vertices in $G_n$ as required.

If $S$ is of type (ii), with $xy$ coloured $c$, then
$W_x, W_y\subseteq \bigcup_{d\in D_{S'}} V_d$. If $c\in A_z$ then
$|A_z|=3$ and Chernoff's bound implies that $|W_z|\leq n/3+o(n)$ holds with probability $1-o(1)$. Note
that this also holds when $c\notin A_z$ as $W_z=\emptyset$. So $D_S =
W_z\cup\bigcup_{d\in D_{S'}} V_d$ and
\[
|D_S|\leq |W_z|+n|D_{S'}|\leq n/3+o(n)+n(2k/3-1/3) = (2/3+o(1))nk\]
 with probability $1-o(1)$ as claimed.

To complete the proof we need to consider what happens when
$|S'|<3$. This can only occur if in $G_n$ there exists a vertex
class, $V_m$ say (with $m\in V(\hat{G})$), that contains two or three
members of $S$. By the definition of a good set at least two of the
vertices in $S$ are incident with an edge of colour $c$, so at least
one such vertex is present in $V_m$ implying $c\in A_m$. This shows
that we can always choose a set $T$ of three vertices in $\hat{G}$ that
contains both $S'$ and a vertex $u$ (possibly contained in $S'$) with the property that $um$ is
coloured $c$. The fact that $S'$ is a subset of $T$ gives us
\[
D_S \subseteq \left(\bigcup_{d\in D_T} V_d\right) \cup
\left(\bigcup_{v\in S} W_v\right),
\]
where $D_T$ is the set strongly $c$-dominated by $T$ in $\hat{G}$,
and having $u\in T$ ensures $W_v\subseteq \bigcup_{d\in D_T} V_d$
for every $v\in S\cap V_m$. Consequently we can apply the same
argument we used for $|S'|=3$. We note that (since $\hat{G}$ is a
counterexample) $|D_{T}|\leq 2k/3-1/3$, and that if $S$ is of type
(i) we have $D_S \subseteq \bigcup_{d\in D_T} V_d$
otherwise $S$ is of type (ii) and $D_S \subseteq
W_z\cup\bigcup_{d\in D_T} V_d$. In either case we get the desired
result that $D_S$ contains at most $2/3+o(1)$ of the vertices with
high probability.
\end{proof}

We note that there are other types of $3$ vertex sets that we
could potentially utilize other than the ``good sets'', however our
proof does not require them and so we will not discuss them here.
Also similar results can be proven for sets larger than $3$ which
may be of use in other problems. For a more general treatment we
refer the reader to Kr\'al', Liu, Sereni, Whalen, and Yilma
\cite{Kral_4Set}.

%
%

Given two $3$-coloured complete graphs $F$, $G$ with $|V(F)|\leq
|V(G)|$, we define $d_F(G)$, the \emph{density} of $F$ in $G$, to be
the proportion of sets of size $|V(F)|$ in $G$ that induce a
$3$-coloured complete graph that is identical to $F$ up to a
re-ordering of vertices.

\begin{figure}[tbp]
\begin{center}
\ifdefined\ColourDiagrams
\includegraphics[height=5cm]{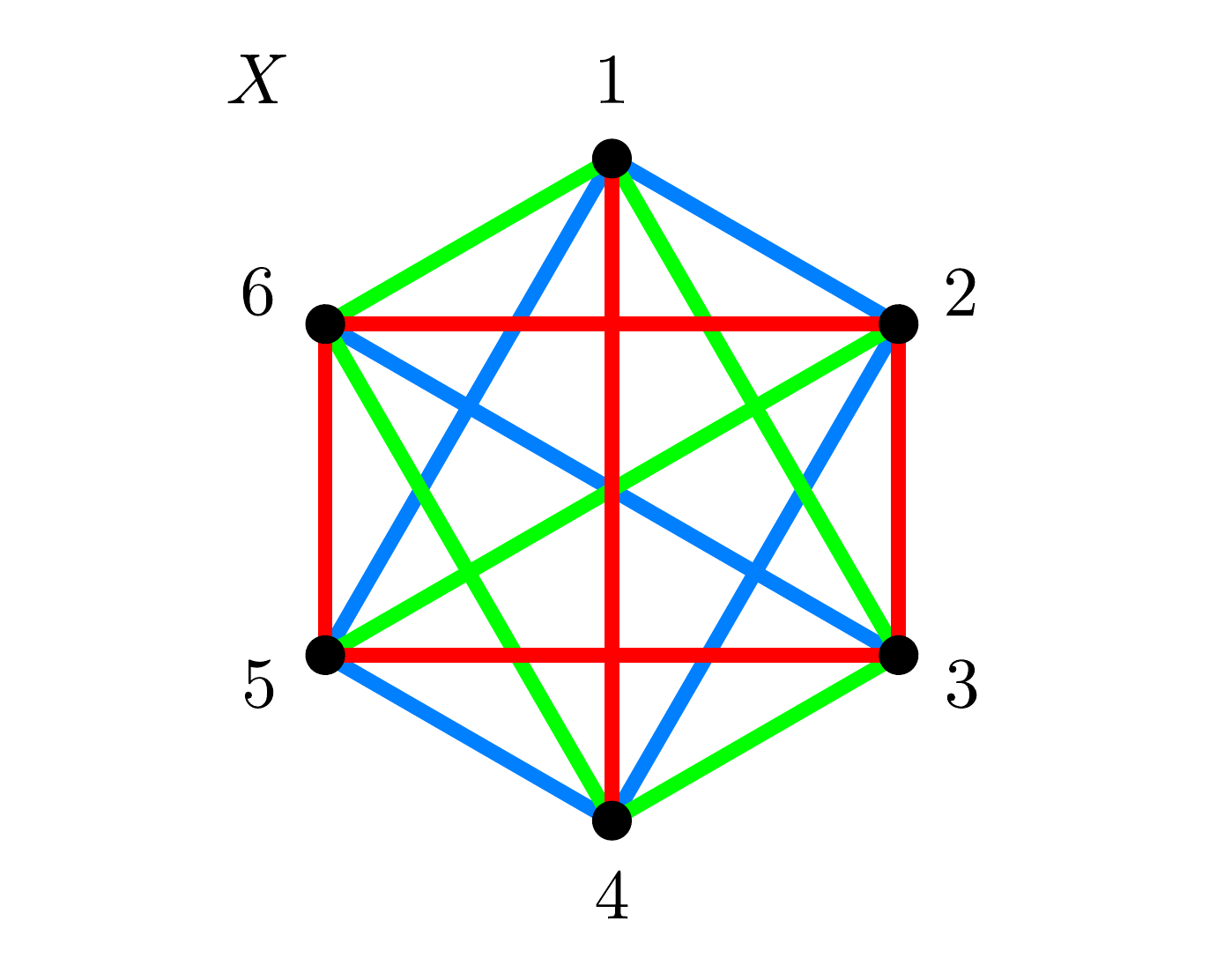}
\else
\includegraphics[height=5cm]{GraphX_grey}
\fi \caption{The $3$-coloured graph $X$.}\label{FIG:GraphX}
\end{center}
\end{figure}

In \cite{Kral_4Set} the authors bound the density in $G_n$ of a
family of graphs in order to contradict Lemma \ref{LEM:HasRainbow}.
In particular they chose their family to consist of all $3$-coloured
complete graphs on $5$ vertices that contain a vertex $v$ with
$|A_v|=3$. We will instead bound the density of a single $6$ vertex
graph $X$, whose coloured edge sets are given by
\[
\{14, 23, 35, 56, 62\}, \{25, 34, 46, 61, 13\}, \{36, 45, 51, 12,
24\},
\]
see Figure \ref{FIG:GraphX}.

Observe that for our counterexample $\hat{G}$, Lemma
\ref{LEM:HasRainbow} implies that there exists a vertex $u$ with
$|A_u|=3$, which in turn implies there exists a class of $n$
vertices $V_u$ in $G_n$ with the edges coloured uniformly at random.
By considering the probability of finding $X$ in $V_u$ we have the
following simple bound for $d_X(G_n)$.

%

\begin{corollary}\label{COR:HasRainbow}
With probability $1-o(1)$,
\[
d_X(G_n) \geq k^{-|V(X)|} 3^{-|E(X)|}+o(1).
\]
\end{corollary}

By encoding Lemma \ref{LEM:GoodSet} using flag algebras we will show
that with high probability, $d_X(G_n)= o(1)$,  a contradiction proving that no
counterexample exists.

\subsection{Flag algebras}\label{SEC:FlagAlgebras}
Razborov's semidefinite flag algebra method introduced in \cite{RF}
and \cite{R4} has proven to be an invaluable tool in extremal graph
theory. Many results have been found through its application, see
for example \cite{B11}, \cite{BT}, \cite{BTNew}, \cite{FV},
\cite{Kral_Digraphs}, \cite{Kral_4Set}. We also refer interested
readers to \cite{Baber_Cube} for a minor improvement to the general
method. Our notation and description of the method for $3$-coloured
graphs is largely adapted from the explanation given by Baber and
Talbot in \cite{BT}.

We will say that two $3$-coloured complete graphs are
\emph{isomorphic} if they can be made identical by permutating their
vertices. Let $\mc{H}$ be the family of all $3$-coloured complete
graphs on $l$ vertices, up to isomorphism. If $l$ is sufficiently
small we can explicitly determine $\mc{H}$ (by computer search if
necessary). For $H\in \mc{H}$ and a large $3$-coloured complete
graph $K$, we define $p(H;K)$ to be the probability that a random
set of $l$ vertices from $K$ induces a $3$-coloured complete graph
isomorphic to $H$.

Using this notation and averaging over $l$ vertex sets in $G_n$
(with $l\geq |V(X)|$), we can show
\begin{equation}\label{EQN:DXTrivialDensity}
d_X(G_n) = \sum_{H\in \mc{H}}d_X(H)p(H;G_n),
\end{equation}
and hence $d_X(G_n)\leq \max_{H\in\mc{H}}d_X(H)$. This bound is
unsurprisingly extremely poor. We will rectify this by creating a
series of inequalities from Lemma \ref{LEM:GoodSet} that we can use
to improve (\ref{EQN:DXTrivialDensity}). To do this we first need to
consider how small pairs of $3$-coloured complete graphs can
intersect. We will use Razborov's method and his notion of flags and
types to formally do this.

A \emph{flag}, $F=(K',\theta)$, is a $3$-coloured complete graph
$K'$ together with an injective map $\theta: \{1,\ldots,s\}\to
V(K')$. If $\theta$ is bijective (and so $|V(K')|=s$) we call the
flag a \emph{type}. For ease of notation given a flag
$F=(K',\theta)$ we define its order $|F|$ to be $|V(K')|$. Given a
type $\sigma$ we call a flag $F=(K',\theta)$ a $\sigma$-\emph{flag}
if the induced labelled $3$-coloured subgraph of $K'$ given by
$\theta$ is $\sigma$.

For a type $\sigma$ and an integer $m\geq |\sigma|$, let $\msig$ be
the set of all $\sigma$-flags of order $m$, up to isomorphism. For a
non-negative integer $s$ and $3$-coloured complete graph $K$, let
$\Theta(s,K)$ be the set of all injective functions from
$\{1,\ldots,s\}$ to $V(K)$. Given $F\in\msig$ and
$\theta\in\Theta(|\sigma|,K)$ we define $p(F,\theta;K)$ to be the
probability that an $m$-set $V'$ chosen uniformly at random from
$V(K)$ subject to $\tr{im}(\theta)\subseteq V'$, induces a
$\sigma$-flag $(K[V'],\theta)$ that is isomorphic to $F$.

If $F_1\in \F_{m_1}^\sigma$, $F_2\in \F_{m_2}^\sigma$, and
$\theta\in \Theta(|\sigma|,K)$ then $p(F_1,\theta;K)p(F_2,\theta;K)$
is the probability that two sets $V_1,V_2\subseteq V(K)$ with
$|V_1|=m_1$, $|V_2|=m_2$, chosen independently at random subject to
$\tr{im}(\theta)\subseteq V_1\cap V_2$, induce $\sigma$-flags
$(K[V_1],\theta)$, $(K[V_2],\theta)$ that are isomorphic to $F_1,
F_2$ respectively. We define the related probability,
$p(F_1,F_2,\theta;K)$, to be the probability that two sets
$V_1,V_2\subseteq V(K)$ with $|V_1|=m_1$, $|V_2|=m_2$, chosen
independently at random subject to $\tr{im}(\theta)=V_1\cap V_2$,
induce $\sigma$-flags $(K[V_1],\theta)$, $(K[V_2],\theta)$ that are
isomorphic to $F_1, F_2$ respectively. Note that the difference in
definitions between $p(F_1,\theta;K)p(F_2,\theta;K)$ and
$p(F_1,F_2,\theta;K)$ is that of choosing the two sets with or
without replacement. It is easy to show that
$p(F_1,\theta;K)p(F_2,\theta;K)=p(F_1,F_2,\theta;K)+o(1)$ where the
$o(1)$ term vanishes as $|V(K)|$ tends to infinity.

Taking the expectation over a uniformly random choice of $\theta\in
\Theta(|\sigma|,K)$ gives
\begin{equation}\nonumber
\mathbf{E}_{\theta\in\Theta(|\sigma|,K)}\left[p(F_1,\theta;K)p(F_2,\theta;K)\right]
=
\mathbf{E}_{\theta\in\Theta(|\sigma|,K)}\left[p(F_1,F_2,\theta;K)\right]+o(1).
\end{equation}
Furthermore the expectation on the right hand side can be rewritten
in terms of $p(H;K)$ by averaging over $l$-vertex subgraphs of $K$,
provided $m_1+m_2-|\sigma|\leq l$ (i.e.\ $F_1$ and $F_2$
intersecting on $\sigma$ fits inside an $l$ vertex graph). Hence
\begin{multline}\label{EQN:FlagProduct}
\mathbf{E}_{\theta\in\Theta(|\sigma|,K)}\left[p(F_1,\theta;K)p(F_2,\theta;K)\right]
= \\
\sum_{H\in\mc{H}}\mathbf{E}_{\theta\in\Theta(|\sigma|,H)}\left[p(F_1,F_2,\theta;H)\right]p(H;K)+o(1).
\end{multline}
Observe that the right hand side of (\ref{EQN:FlagProduct}) is a
linear combination of $p(H;K)$ terms whose coefficients can be
explicitly calculated using just $\mc{H}$, this will prove useful as
(\ref{EQN:DXTrivialDensity}) is of a similar form.

Given $\F_m^\sigma$ with $2m-|\sigma|\leq l$ and a positive
semidefinite matrix $Q=(q_{ab})$ of dimension $|\mc{F}^\sigma_m|$,
let $\mathbf{p}_\theta=(p(F,\theta;K):F\in \mc{F}^\sigma_m)$ for
$\theta\in\Theta(|\sigma|,K)$. Using (\ref{EQN:FlagProduct}) and the
linearity of expectation we have
\begin{align}\label{EQN:SingleQ}
0\leq
\mathbf{E}_{\theta\in\Theta(|\sigma|,K)}[\mathbf{p}_\theta^TQ\mathbf{p}_\theta]
&= \sum_{H\in\mc{H}} a_H(\sigma,m,Q) p(H;K)+o(1)
\end{align}
where
\[
a_H(\sigma,m,Q) = \sum_{F_a,F_b\in
\mc{F}^\sigma_m}q_{ab}\mathbf{E}_{\theta\in
\Theta(|\sigma|,H)}[p(F_a,F_b,\theta;H)].
\]
Note that $a_H(\sigma,m,Q)$ is independent of $K$ and can be
explicitly calculated. Combining (\ref{EQN:SingleQ}) when $K=G_n$
with (\ref{EQN:DXTrivialDensity}) gives
\begin{align*}
d_X(G_n)&\leq\sum_{H\in\mc{H}}(d_X(H)+a_H(\sigma,m,Q))
p(H;G_n)+o(1)\\
&\leq \max_{H\in\mc{H}}(d_X(H)+a_H(\sigma,m,Q))+o(1).
\end{align*}
Since some of the $a_H(\sigma,m,Q)$ values may be negative (for a
careful choice of $Q$) this may be a better bound (asymptotically)
for $d_X(G_n)$. To help us further reduce the bound we can of course
create multiple inequalities of the form given by
(\ref{EQN:SingleQ}) by choosing different types $\sigma_i$, orders
of flags $m_i$, and positive semidefinite matrices $Q_i$. Let
$\alpha_H = \sum_i a_H(\sigma_i,m_i,Q_i)$ and hence we can say
$d_X(G_n)\leq \max_{H\in\mc{H}}(d_X(H)+\alpha_H)+o(1)$. Finding the
optimal choice of matrices $Q_i$ which lowers the bound as much as
possible is a convex optimization problem, in particular a
semidefinite programming problem. As such we can use freely
available software such as CSDP \cite{B} to find the $Q_i$.

So far the bound on $d_X(G_n)$ is valid for any $3$-coloured
complete graph; we have not yet made any use of the fact that $G_n$
comes from our counterexample $\hat{G}$. Kr\'al', Liu, Sereni,
Whalen, and Yilma remedy this, see Lemma 3.3 in \cite{Kral_4Set}, by
constructing a small set of constraints that $G_n$ must satisfy but
a general $3$-coloured complete graph may not. By using an idea of
Hladky, Kr\'al', and Norine \cite{Kral_Digraphs} we can
significantly increase the number of such constraints.

We say that a $\sigma$-flag $F$ is \emph{$c$-good} if the colouring of $F$ and the size of $\sigma$ imply that $\sigma$ is a good set for $c$ in $F$.

\begin{lemma}\label{LEM:GkFlagConstraint}
Given a colour $c$ and a $c$-good $\sigma$-flag $F$, the following holds with probability $1-o(1)$.
\[
\mathbf{E}_{\theta\in\Theta(|\sigma|,G_n)}\left[p(F,\theta;G_n)\left(\frac{2}{3}p(\sigma,\theta;G_n)
-\sum_{F'\in\mc{D}}p(F',\theta;G_n)\right)\right]+o(1) \geq 0,
\]
where $\mc{D}\subseteq\mc{F}^\sigma_{|\sigma|+1}$ is the set of all
$\sigma$-flags on $|\sigma|+1$ vertices where the vertex not in
$\sigma$ is $c$-dominated by the type.
\end{lemma}

We note that when $F = \sigma$ Lemma \ref{LEM:GkFlagConstraint} is
equivalent to Lemma 3.3 in \cite{Kral_4Set}.

\begin{proof}
For a fixed $\theta\in\Theta(|\sigma|,G_n)$ if $p(F,\theta;G_n)=0$
then trivially we get
\[
p(F,\theta;G_n)\left(\frac{2}{3}p(\sigma,\theta;G_n)
-\sum_{F'\in\mc{D}}p(F',\theta;G_n)\right) \geq 0.
\]
If $p(F,\theta;G_n)>0$ then there exists a copy of $F$ in $G_n$ and so the
image of $\theta$ is $\sigma$ (or equivalently
$p(\sigma,\theta;G_n)=1$) and $\sigma$ must be a good set for $c$. By
Lemma \ref{LEM:GoodSet} we know that with probability $1-o(1)$,
\[
\frac{2}{3}+o(1)\geq \sum_{F'\in\mc{D}}p(F',\theta;G_n),
\]
which implies
\[
p(F,\theta;G_n)\left(\frac{2}{3}p(\sigma,\theta;G_n)
-\sum_{F'\in\mc{D}}p(F',\theta;G_n)\right)+o(1) \geq 0.
\]
Taking the expectation completes the proof.
\end{proof}

Given a $c$-good flag $F$,
equation (\ref{EQN:FlagProduct}) tells us that provided $|F|+1\leq
l$ we can express the inequality given in Lemma
\ref{LEM:GkFlagConstraint} as
\begin{align}\label{EQN:Gk}
\sum_{H\in\mc{H}}b_H(c,F)p(H;G_n)+o(1)\geq 0,
\end{align}
where $b_H(c,F)$ can be explicitly calculated from $c,F$, and $H$.
Equation (\ref{EQN:Gk}) is of the same form as (\ref{EQN:SingleQ})
and as such we can use it in a similar way to improve the bound on
$d_X(G_n)$. Moreover observe that we can multiply (\ref{EQN:Gk}) by
any non-negative real value without changing its form.

Let $\mc{C}$ be a set of pairs of colours $c$ and $c$-good flags $F$
satisfying $|F|+1\leq l$. For $(c,F)\in\mc{C}$ let $\mu(c,F)\geq 0$
be a real number (whose value we will choose later to help us
improve the bound on $d_X(G_n)$). To ease notation we define
$\beta_H = \sum_{(c,F)\in\mc{C}} \mu(c,F)b_H(c,F)$. It is easy to
check $\sum_{H\in\mc{H}} \beta_H p(H;G_n)+o(1)\geq 0$, thus
combining it with (\ref{EQN:DXTrivialDensity}) and terms such as
(\ref{EQN:SingleQ}) gives
\begin{align*}
d_X(G_n)
&\leq\sum_{H\in\mc{H}}(d_X(H)+\alpha_H+\beta_H)p(H;G_n)+o(1)\\
&\leq\max_{H\in\mc{H}}(d_X(H)+\alpha_H+\beta_H)+o(1)
\end{align*}
Finding an optimal set of non-negative coefficients $\mu(c,F)$ and
semidefinite matrices $Q_i$ can still be posed as a semidefinite
programming problem.

We complete the proof of Theorem \ref{THM:MainResult} with the
following lemma which contradicts Corollary \ref{COR:HasRainbow}.

\begin{lemma}\label{LEM:DXIsZero}
With probability $1-o(1)$ we have $d_X(G_n)=o(1)$.
\end{lemma}

\begin{proof}
By setting $l=6$ (the order of the graphs $H$) and solving a
semidefinite program we can find coefficients $\mu(c,F)$ and
semidefinite matrices $Q_i$ such that
$\max_{H\in\mc{H}}(d_X(H)+\alpha_H+\beta_H)=0$. The relevant data
needed to check this claim can be found in the data file
\texttt{2-3.txt}. There is too much data to check by hand so we also
provide the C++ program \textsf{DominatingDensityChecker} to check
the data file. Both data file and proof checker may be downloaded from the arXiv \url{http://arxiv.org/e-print/1306.6202v1}.

It is worth noting that in order to get a tight bound we used the
methods described in Section 2.4.2 of \cite{B11} to remove the
rounding errors from the output of the semidefinite program solvers.

We end by mentioning that when $l=6$ the computation has to consider
$25506$ non-isomorphic graphs which form $\mc{H}$, and as a result
solving the semidefinite program is very time consuming. However,
our method for proving the result makes no preferences between the
colours. Consequently it is quite easy to see that if there exists a
solution then there must also exist a solution which is invariant
under the permutations of the colours. So if $H_1,H_2\in \mc{H}$ are
isomorphic after a permutation of their colours then in an
``invariant solution'' $d_X(H_1)+\alpha_{H_1}+\beta_{H_1} =
d_X(H_2)+\alpha_{H_2}+\beta_{H_2}$ must necessarily hold. Therefore
by restricting our search to invariant solutions we need only worry
about those $H$ in $\mc{H'}$ the set of $3$-coloured complete graphs
on $l$ vertices that are non-isomorphic even after a permutation of
colours. For $l=6$, $|\mc{H'}|=4300$ which results in a
significantly easier computation.
\end{proof}

\section{Open problems}

\begin{figure}[tbp]
\begin{center}
\ifdefined\ColourDiagrams
\includegraphics[height=6.5cm]{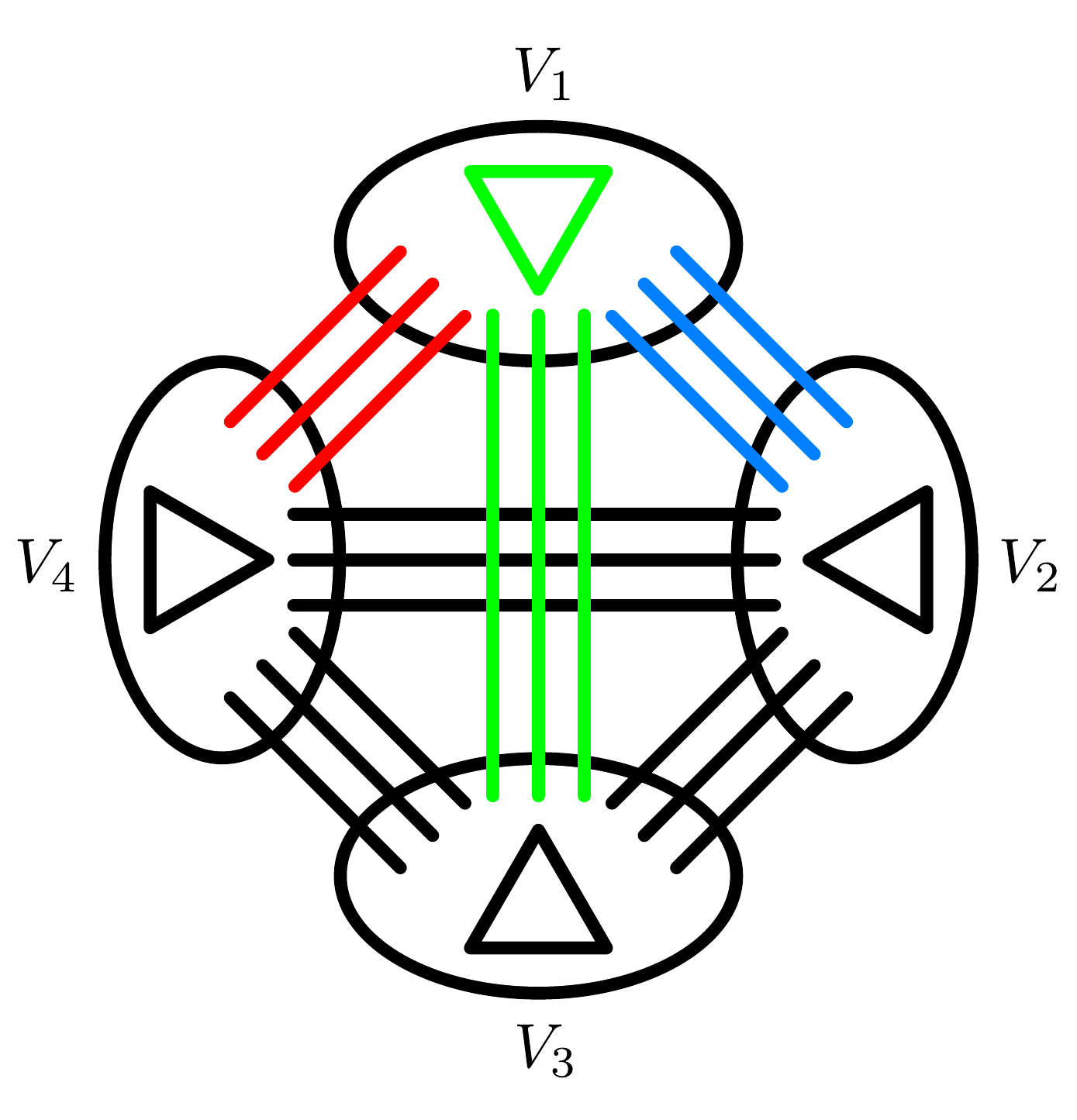}
\caption{A generalisation of Kierstead's construction to $4$
colours. The classes $V_1,V_2,V_3$, and $V_4$ contain $2/5,1/5,1/5$,
and $1/5$ of the vertices respectively.}
\else
\includegraphics[height=6.5cm]{Extremal4_grey}
\caption{A generalisation of Kierstead's construction to $4$
colours. The $4$ colours are represented by solid black, dashed
black, solid grey, and dashed grey lines. The classes $V_1,V_2,V_3$,
and $V_4$ contain $2/5,1/5,1/5$, and $1/5$ of the vertices
respectively.} \fi \label{FIG:Extremal4}
\end{center}
\end{figure}

Erd\H os, Faudree, Gould, Gy\'arf\'as, Rousseau, and Schelp ask in
\cite{Erdos_22Set} whether every $4$-coloured complete graph always
contains a small set of vertices that monochromatically dominate at
least $3/5$ of the vertices. The value of $3/5$ comes from
considering the $4$ colour equivalent of Kierstead's construction
given in Figure \ref{FIG:Extremal4}. It shows that we cannot hope
to find a small set of vertices that monochromatically dominate
significantly more than $3/5$ of the vertices. Also regardless of
the size of the dominating set we can at most guarantee that $\lceil
3n/5\rceil$ vertices will be strongly monochromatically dominated in
an $n$ vertex graph.

By applying Chernoff's bound it is easy to see that a typical random
$4$-colouring on an $n$ vertex graph contains no $3$-sets that
monochromatically dominate more than $(1-(3/4)^3)n+o(n)$ vertices,
which is less than $3n/5$ when $n$ is large. So the minimal possible
dominating set size is $4$.

\begin{figure}[tbp]
\begin{center}
\ifdefined\ColourDiagrams
\includegraphics[height=4cm]{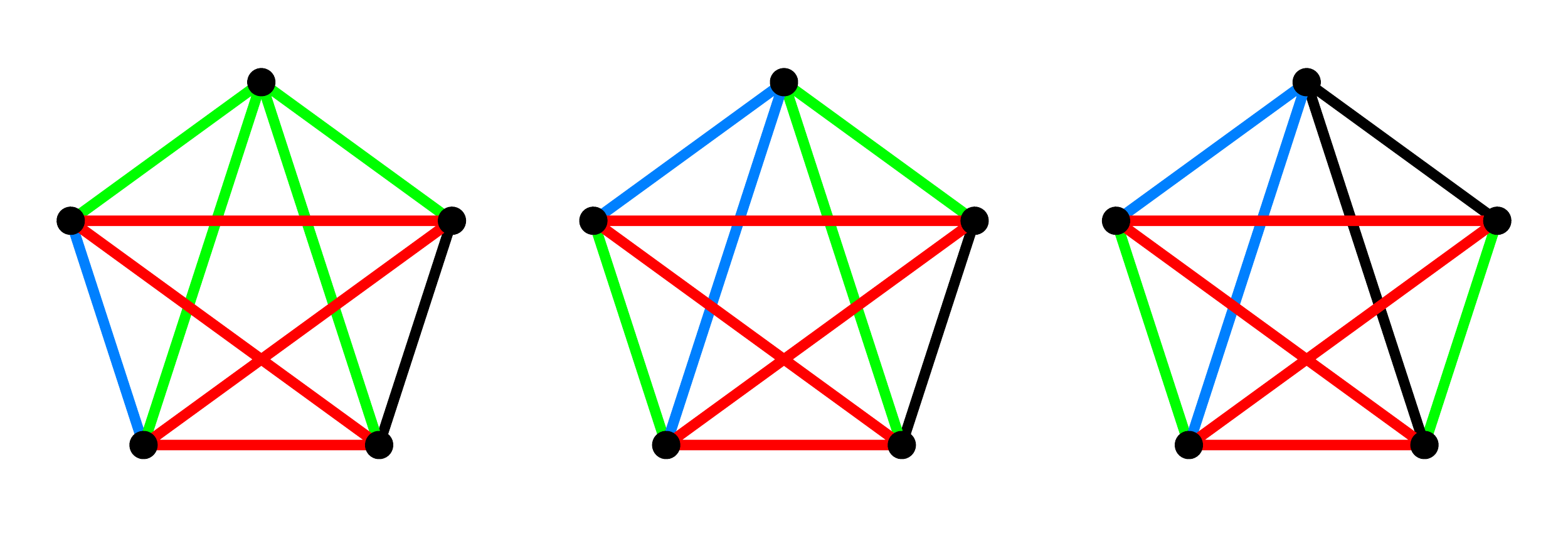}
\caption{Canonical members of a family of $4$-coloured complete
graphs that replace $X$.}
\else
\includegraphics[height=4cm, width=13cm]{5Vertex_grey}
\caption{Canonical members of a family of $4$-coloured complete
graphs that replace $X$. The $4$ colours are represented by solid
black, dashed black, solid grey, and dashed grey lines.}
\fi
\label{FIG:5Vertex}
\end{center}
\end{figure}

We could not prove that there always exists a $4$-set that strongly
monochromatically dominates $3/5$ of the vertices in a complete
$4$-coloured graph. However, by generalising the method given in
Section \ref{SEC:ThmProof}, and replacing $X$ with a specific family
of graphs, we were able to show that there exist $4$-sets that
strongly monochromatically dominate $0.5711$ of the vertices. The family of
graphs we chose to bound instead of $X$ are the $48$ graphs on $5$
vertices that are, up to a permutation of colours, isomorphic to one
of those given in Figure \ref{FIG:5Vertex}. It is not immediately
obvious why such a family should have a positive density in a
counterexample, so we will outline why this is the case.

\begin{lemma}\label{LEM:4Colour}
Any $4$-coloured complete graph $\hat{G}$ on $k$ vertices that has the property that
every set of $3$ vertices strongly $c$-dominates strictly less than
$3/5$ of the vertices for every colour $c$, must contain either
\begin{enumerate}
\item[(i)] a vertex $u$ with $|A_u|=4$, or
\item[(ii)] two vertices $v, w$ with $|A_v|=|A_w|=3$ and $A_v\neq A_w$.
\end{enumerate}
\end{lemma}

\begin{proof}
Let the set of colours be $\{1,2,3,4\}$. Trivially $\hat{G}$ cannot
contain a vertex $u$ with $|A_u|=1$, otherwise any set containing
$u$ will strongly dominate all the vertices. It is therefore sufficient
to show that no $\hat{G}$ can exist in which every vertex $u$ satisfies
$|A_u|=2$, or $A_u = \{1,2,3\}$.

We can take our vertices and partition them into disjoint classes
based on their value of $A_v$. For ease of notation we will refer to
these classes by $V_S$ where $S$ is a subset of the colours, for
example $V_{13}$ contains all the vertices $v$ which have $A_v =
\{1,3\}$. We will split our argument into multiple cases depending
on whether or not a class is empty. Throughout we will make use of
the fact that $\hat{G}$ cannot be $3$-coloured otherwise by Theorem
\ref{THM:MainResult} we can find a $3$-set which dominates over
$3/5$ of the vertices. Also note that if $S,T\subseteq\{1,2,3,4\}$
and $S\cap T=\emptyset$ then either $V_S$ or $V_T$ must be empty as
any edge that goes between the two classes must have a colour in
$S\cap T$.

Suppose $V_{123}=\emptyset$. Without loss of generality we can
assume $V_{12}\neq\emptyset$ (implying $V_{34}=\emptyset$). There
must be another non-empty class otherwise $\hat{G}$ is $3$-coloured.
Without loss of generality we may assume $V_{13}\neq\emptyset$
(implying $V_{24}=\emptyset$). To avoid being $3$-coloured we must
have $V_{14}\neq\emptyset$ (implying $V_{23}=\emptyset$). There are
no more classes we could add and $\hat{G}$ has all its vertices
strongly $1$-dominated by a $2$-set containing a vertex from
$V_{12}$ and a vertex from $V_{13}$ which is a contradiction.

Suppose $V_{123}\neq\emptyset$. To avoid being $3$-coloured at least
one of $V_{14}, V_{24}, V_{34}$ must be non-empty. Without loss of
generality assume $V_{14}\neq\emptyset$ (implying
$V_{23}=\emptyset$). To avoid having every vertex strongly
$1$-dominated by a $2$-set containing a vertex from $V_{123}$ and a vertex from $V_{14}$, either $V_{24}$ or $V_{34}$ must be non-empty. Without
loss of generality assume $V_{24}\neq\emptyset$ (implying
$V_{13}=\emptyset$). If $V_{34}=\emptyset$, then the vertices in
$\hat{G}$ are partitioned into three disjoint classes $V_{123}\cup
V_{12}$, $V_{14}$, and $V_{24}$. We can strongly $c$-dominate at
least $2/3$ of the vertices by choosing $c$ to be the colour of the
edges that go between the largest two of the classes and by choosing
our dominating set to contain a vertex from each of the largest two
classes.

The only case left to consider is when $V_{123}\neq\emptyset$,
$V_{14}\neq\emptyset$, $V_{24}\neq\emptyset$, and
$V_{34}\neq\emptyset$. (Note that $\hat{G}$ resembles Figure
\ref{FIG:Extremal4}.) We may suppose that $|V_{14}|\geq |V_{24}|\geq |V_{34}|$ and so either $|V_{14}\cup V_{24}\cup V_{34}|\geq 3k/5$ (where $k$ is the order of $\hat{G}$) and a $2$-set containing a vertex from each of $V_{14}$ and $V_{24}$ will strongly $4$-dominate at least $3/5$ of the vertices or $|V_{123}\cup V_{14}|\geq 3k/5$ and a $2$-set containing a vertex from each of $V_{123}$ and $V_{14}$ will strongly $1$-dominate at least $3/5$ of the vertices. 
\end{proof}

\begin{corollary}
Let $\hat{G}$ be a $4$-coloured complete graph on $k$ vertices that
has the property that every set of $3$ vertices strongly
$c$-dominates strictly less than $3/5$ of the vertices for every
colour $c$. If $G_n$ is constructed from $\hat{G}$ as before then, with probability $1-o(1)$,
\[
d_F(G_n)\geq k^{-|V(F)|}4^{-|E(F)|}+o(1)
\]
holds for some graph $F$ that is, up to a permutation of colours,
isomorphic to one of the graphs given in Figure \ref{FIG:5Vertex}.
\end{corollary}

\begin{proof}
By Lemma \ref{LEM:4Colour} $\hat{G}$ must have a vertex $u$ with
$|A_u|=4$ or two vertices $v,w$ with $|A_v|=|A_w|=3$ and $A_v\neq
A_w$. If we have a vertex $u$ with $|A_u|=4$ there will be a vertex
class of size $n$ in $G_n$ with all its edges coloured uniformly at
random. The result trivially holds by considering the density of any
$5$ vertex graph inside this vertex class.

Suppose instead there exist two vertices $v,w$ in $\hat{G}$ with
$|A_v|=|A_w|=3$ and $A_v\neq A_w$. To ease notation let $c_{xy}$ be
the colour of the edge $xy$. Note that by the definition of
$\hat{G}$ we know that every vertex is not strongly
$c_{vw}$-dominated by the set $\{v,w\}$. Consequently there must
exist a vertex $z$ such that $c_{vz}\neq c_{vw}$ and $c_{wz}\neq
c_{vw}$. Now consider the vertex classes $V_v$, $V_w$, and $V_z$ in
$G_n$. There are $9$ possible $5$ vertex graphs that could be formed
from taking one vertex in $V_z$ and two vertices in $V_v$ and $V_w$.
Only one of these $9$ graphs has the property that the two vertices
$v_1, v_2$ chosen from $V_v$ satisfy $A_{v_1}=A_{v_2}=A_v$ and the
two vertices $w_1, w_2$ chosen from $V_w$ satisfy
$A_{w_1}=A_{w_2}=A_w$. This graph is, up to a permutation of
colours, isomorphic to one of those given in Figure
\ref{FIG:5Vertex}. The result trivially follows by considering its
density in $V_v\cup V_w\cup V_z$.
\end{proof}

Although our discussion has centred on $4$-colourings, it would also
be interesting to know what happens for complete graphs which are
$r$-coloured for $r\geq 5$.

\end{document}